\documentclass[preprint, 11pt, english]{elsarticle}
\usepackage{amsmath}
\usepackage{latexsym, amssymb}
\usepackage{amsthm}
\usepackage{txfonts,xcolor}

\newtheorem{thm}{Theorem}[section] %the resolution could also be [subsection]

\newtheorem{cor}[thm]{Corollary}

\newtheorem{lem}[thm]{Lemma}

\newtheorem{ques}[thm]{Question}

\theoremstyle{definition}
\newtheorem{rem}[thm]{Remark}
\newtheorem{exmpl}[thm]{Example}

\newcommand\operA[2]{{\if!#2!\operatorname{#1}\else{\operatorname{#1}_{#2}^{\phantom{I}}}\fi}} % To be used within Bdefs. Usage: $\operA{N}{K/F}$ produces $N_{K/F}$; $\operA{N}{}$ produces $N$.

\newcommand\Cref[1]{{Corollary~\ref{#1}}}%

\def\tr{{\operatorname{Tr}}}

\def\norm{{\operatorname{Norm}}}

\newcommand{\Trace}[1][]{\if!#1!\operatorname{Tr}\else{\operatorname{Tr}_{#1}^{\phantom{I}}}\fi} % Usage: $\Tr[K/F](a)$.

\long\def\forget#1\forgotten{{}} %

\def\({\left(}
\def\){\right)}

\newif\iffurther
\furtherfalse
\newif\ifXY % turns XY version on/off
\XYtrue     % Turn it on
\ifXY

\input xy
\input xyidioms.tex
\usepackage{xy}
\xyoption{all} %
\fi % For \ifXY

\usepackage{babel}

\makeatletter
\def\ps@pprintTitle{%
	\let\@oddhead\@empty
	\let\@evenhead\@empty
	\def\@oddfoot{}%
	\let\@evenfoot\@oddfoot}
\makeatother

\begin{document}

\begin{frontmatter}

\title{Fixed Points of Polynomials over Division Rings}
\author{Adam Chapman}
\ead{adam1chapman@yahoo.com}
\address{Department of Computer Science, Academic College of Tel-Aviv-Yaffo, Rabenu Yeruham St., P.O.B 8401 Yaffo, 6818211, Israel}
\author{Solomon Vishkautsan}
\ead{wishcow@gmail.com}
\address{Department of Computer Science, Tel-Hai Academic College, Upper Galilee, 12208 Israel}

\begin{abstract}
We study the discrete dynamics of standard (or left) polynomials $f(x)$ over division rings $D$. We define their fixed points to be the points $\lambda \in D$ for which $f^{\circ n}(\lambda)=\lambda$ for any $n \in \mathbb{N}$, where  $f^{\circ n}(x)$ is defined recursively by $f^{\circ n}(x)=f(f^{\circ (n-1)}(x))$ and $f^{\circ 1}(x)=f(x)$. Periodic points are similarly defined. We prove that $\lambda$ is a fixed point of $f(x)$ if and only if $f(\lambda)=\lambda$, which enables the use of known results from the theory of polynomial equations, to conclude that any polynomial of degree $m \geq 2$ has at most $m$ conjugacy classes of fixed points. We also consider arbitrary periodic points, and show that in general, they do not behave as in the commutative case. We provide a sufficient condition for periodic points to behave as expected.  
\end{abstract}


\end{frontmatter}

\section{Introduction}

Arithmetic dynamics is a growing area of mathematics that has received a significant amount of attention in recent years (see \cite{Silverman:2007}).
The main focus of previous works have been fixed, periodic and pre-periodic points on polynomials and rational functions over fields, and in particular over the rational numbers and $p$-adic numbers.

The goal of this paper is to make the first steps in the direction of arithmetic dynamics in the non-commutative setting. We focus on (standard) polynomials over division rings.
Since polynomials over division rings are not quite functions, and composition is not associative, the definition of fixed points is subtler than in the commutative case (see Section \ref{Fixed}): a polynomial $f(x)$ over a division ring $D$ has a fixed point $\lambda \in D$ if $f^{\circ n}(\lambda)=\underbrace{f \circ \dots \circ f}_{n \ \text{times}}(\lambda)=\lambda$ for any $n \in \mathbb{N}$. The main result of this paper is that $\lambda$ is a fixed point of $f(x)$ if and only if $f(\lambda)=\lambda$ (Corollary~\ref{cor:fixed}), which reduces the problem of finding fixed points to solving a standard polynomial equation, a problem that has been fully solved in the literature when $D$ is a quaternion algebra, and even for more general central division algebras a lot can be said (see \cite{JanovskaOpfer2010} for the case of $D=\mathbb{H}$ and \cite{ChapmanMachen:2017} for a more general treatment).
We also consider arbitrary $r$-periodic points, i.e., points for which $f^{\circ nr}(\lambda)=\lambda$ for any $n \in \mathbb{N}$, and show that in general $f^{\circ r}(\lambda)=\lambda$ does not imply that $\lambda$ is an $r$-periodic point. We are able to prove, however, that the latter implication is true under the condition that $\lambda$ commutes with its orbit (Theorem~\ref{thm:periodic}). 

\section{Division Rings}

An associative unital ring $D$ is a division ring if all nonzero elements in $D$ are invertible.
When $D$ is commutative, it is simply a field, but there are plenty of examples of non-commutative division rings.
The center $F=Z(D)$ of $D$ is nevertheless a field, and $D$ is endowed with the structure of an $F$-vector space.
Therefore, one can consider the dimension $\dim_F D$ of $D$ over $F$.
In case $\dim_F D<\infty$, $D$ is called a ``central division $F$-algebra". In this case, $\dim_F D$ must be $d^2$ for some positive integer $d$, and $D\otimes_F K \cong M_d(K)$ for any maximal subfield $K$ of $D$ (where $M_d(K)$ is the ring of square matrices of order $d$ over the field $K$).
This $d$ is called the ``degree" of $D$ and denoted $\deg D$. 

The simplest kind of central division $F$-algebras are ``quaternion" division $F$-algebras, which are division algebras of degree 2.
In this case, $D=K \oplus K j$ where $K/F$ is a cyclic quadratic field extension with Galois group $\langle \sigma \rangle$, $j^2=\beta \in F^\times$ and $j t=\sigma(t) j$ for any $t \in K$. The nontrivial automorphism $\sigma$ extends to an involution on $D$ defined by $r+tj \mapsto \overline{r+tj}=\sigma(r)-tj$ for any $r,t \in K$.
This involution is called the ``symplectic" or ``canonical" involution of $D$, and is independent of the choice of a cyclic field extension $K/F$ inside $D$. The ``trace" and ``norm" maps for a quaternion algebra are defined by $\tr(z)=z+\overline{z}$ and $\norm(z)=\overline{z}\cdot z$, the first being linear and the second multiplicative. Every element $\lambda \in D$ thus satisfies $\lambda^2-\tr(\lambda)\lambda+\norm(\lambda)=0$. Moreover, the conjugacy class of $\lambda$ (i.e., the set of all elements $t$ of the form $\mu \lambda \mu^{-1}$) coincides with the elements $t$ satisfying the identity $t^2-\tr(\lambda)t+\norm(\lambda)=0$. For further reading see \cite{BerhuyOggier2013} and \cite{BOI}.

If one takes $F=\mathbb{R}$, $K=\mathbb{C}$ and $\beta=-1$, then the quaternion algebra obtained in this process is Hamilton's quaternion algebra $\mathbb{H}$, which happens to be the unique central division algebra over the field of real numbers.
\section{The Ring of Polynomials} \label{sec:ring-of-polynomials}

Let $D$ be a division ring with center $F$.
The ring of ``standard" polynomials $D[x]$ is defined to be $D \otimes_F F[x]$, i.e., the scalar extension of $D$ to the polynomial ring in one variable over $F$.
The variable $x$ is thus central in $D[x]$, and the latter is an Azumaya algebra when $[D:F] < \infty$.

Each polynomial $f(x) \in D[x]$ can therefore be written with the powers of $x$ placed on the right-hand side and the coefficients placed on the left-hand side (which is why they are occasionally referred to as ``left polynomials"), $f(x)=\sum_{i=0}^m c_i x^i$.
For any $\lambda \in D$, the substitution map $f(x) \mapsto f(\lambda)$ is defined by $f(\lambda)=\sum_{i=0}^m c_i \lambda^i$. Note that this is not a ring homomorphism from $D[x]$ to $D$.
There are however several motivations to study the behavior under substitution, the main one being that $f(x)$ decomposes as $g(x)(x-\lambda)$ if and only if $f(\lambda)=0$, i.e., $(x-\lambda)$ is a linear right factor of $f(x)$ if and only if $\lambda$ is a root of $f(x)$.
This motivated the study of roots of such polynomials. For a complete treatment, see \cite{ChapmanMachen:2017} (and also \cite{LamLeroyOzturk2008}).

We write $f^t(x)$ for the product $\underbrace{f(x)\cdot \ldots\cdot f(x)}_{t \ \text{times}}$ in $D[x]$ for any $t \in \mathbb{N}$.

The substitution map extends to any element from $D[x]$ in a similar way: if $f(x)=\sum_{i=0}^m c_i x^i$ and $g(x)$ are polynomials in $D[x]$, then $f(g(x))=\sum_{i=0}^m c_i (g(x))^i$, which is again a polynomial in $D[x]$.
We define $f^{\circ n}(x)$ recursively by $f(f^{\circ (n-1)}(x))$ for any $n \geq 2$, with $f^{\circ 1}(x)$ being $f(x)$. 
%@@@ We write $f^{*n}(x)$ for $\underbrace{f(\dots(f(x))\dots)}_{n \ \text{times}}$.

Note that not only is this substitution not associative, it is not even power-associative. For example, $f(x)=ax^2+b$ for general $a,b\in D$ does not satisfy $f(f^{\circ 2}(x))=f^{\circ 2}(f(x))$. It does not associate with substitutions of elements from $D$ either, for example if $f(x)=i x^2 \in \mathbb{H}$, then $f(j+1)=2ij$ and $f(2ij)=-4i$, whereas $f^{\circ 2}(x)=-i x^4$, and so $f^{\circ 2}(j+1)=4i$; we conclude $f^{\circ 2}(\lambda)$ is in general not equal to $f(f(\lambda))$ for $\lambda\in{D}$, even though $f^{\circ 2}(x)=f(f(x))$.

\section{Orbits, periodic points and fixed points}\label{Fixed}

Given a polynomial $f(x) \in D[x]$ where $D$ is a division ring, and $\lambda \in D$, the orbit of $\lambda$ under $f(x)$ is the sequence $\{f^{\circ n}(\lambda) : n \in \mathbb{N}\}$. We say that $\lambda$ commutes with its orbit if it commutes with each element in that sequence.
The point $\lambda$ is called ``$r$-periodic" if $f^{\circ (nr)}(\lambda)=\lambda$ for any $n \in \mathbb{N}$.
A fixed point is 1-periodic.

\begin{rem}
Note that if $\lambda$ is a fixed point of $f(x) \in D[x]$, then $\lambda$ commutes with its orbit, because the only element in the orbit is $\lambda$.
\end{rem}

Unfortunately, having $f^{\circ r}(\lambda)=\lambda$ in general does not imply that $f^{\circ(n r)}(\lambda)=\lambda$ for any $n \in \mathbb{N}$, as the following example shows:

\begin{exmpl}
Take $D=\mathbb{H}$ and $f(x)=x^2+(i+1)x+1+ij$ and the point
$$\lambda=-1 + \left(\frac{133}{362}\sqrt{5} - \frac{333}{362}\right)i -\left(\frac{14}{181}\sqrt{5} + \frac{165}{181}\right)j - \left(\frac{26}{181}\sqrt{5} + \frac{22}{181}\right)ij.$$
Then $f^{\circ 2}(\lambda)=\lambda$, but $f^{\circ 4}(\lambda) \neq \lambda$.
\end{exmpl}

We will prove below, however, that the implication is correct when we assume in addition that $f^{\circ t}(\lambda)$ commutes with $\lambda$ for any $1 \leq t <r$.

\begin{lem}
Let $D$ be a division ring, $f(x),g(x) \in D[x]$ and $\lambda\in D$. If $g(\lambda)$ commutes with $\lambda$, then $h(\lambda)=f(\lambda)\cdot g(\lambda)$ where $h(x)=f(x)g(x)$.
\end{lem}

\begin{proof}
Write $f(x)=\sum_{i=0}^m c_i x^i$ and $g(x)=\sum_{j=0}^k d_j x^j$.
Then $h(x)=f(x)\cdot g(x)=\sum_{i=0}^m (c_i x^i)\cdot g(x)=\sum_{i=0}^m \sum_{j=0}^k c_i d_j x^{i+j}$.
Hence, $h(\lambda)=\sum_{i=0}^m \sum_{j=0}^k c_i d_j \lambda^{i+j}\\=\sum_{i=0}^m c_i (\sum_{j=0}^k d_j \lambda^j) \lambda^i=\sum_{i=0}^m c_i (g(\lambda)) \lambda^i=\sum_{i=0}^m c_i \lambda^i g(\lambda)=(\sum_{i=0}^m c_i \lambda^i)g(\lambda)=f(\lambda)\cdot g(\lambda)$.
\end{proof}

\begin{cor}\label{Product}
If $D$ is a division ring, $f(x) \in D[x]$, $\lambda \in D$ and $f(\lambda)$ commutes with $\lambda$, then $f^t(\lambda)=(f(\lambda))^t$ for any $t \in \mathbb{N}$.
\end{cor}

\begin{proof}
By induction on $t$.
Write $g(x)=f^{t-1}(x)$, $h(x)=f(x)g(x)$, and assume $g(\lambda)=(f(\lambda))^{t-1}$, which commutes with $\lambda$.
Then by the previous lemma, $h(\lambda)=f(\lambda)\cdot g(\lambda)=(f(\lambda))^t$.
\end{proof}

In order to facilitate notation, let us denote by $f^{*n}(\lambda)$ the expression \[\underbrace{f(\dots(f(\lambda))\dots)}_{n \ \text{times}}\] for any $n\in\mathbb{N}$ and any $\lambda\in D$. As we mentioned at the end of section~\ref{sec:ring-of-polynomials} , we do not in general have $f^{\circ{n}}(\lambda)=f^{*n}(\lambda)$.
 
\begin{thm}
Let $D$ be a division ring,  $f(x) \in D[x]$, $\lambda \in D$ and $n \in \mathbb{N}$.
Suppose that $\lambda$ commutes with either $\{f^{\circ t}(\lambda) : 1\leq t<n\}$ or $\{f^{*t}(\lambda) : 1\leq t<n\}$. Then $f^{\circ n}(\lambda)=f^{*n}(\lambda)$.
\end{thm}

\begin{proof}
We proceed by induction on $n$ (note that under the inductive hypothesis, the two sets $\{f^{\circ t}(\lambda) : 1\leq t<n\}$ and $\{f^{*t}(\lambda) : 1\leq t<n\}$ are equal).
Write $f(x)=\sum_{i=0}^m c_i x^i$ and $g(x)=f^{\circ(n-1)}(x)=\sum_{j=0}^k d_j x^j$, and assume $g(\lambda)=f^{* (n-1)}(\lambda)$.
Then $f^{\circ n}(x)=f(g(x))=\sum_{i=0}^m c_i (g(x))^i=\sum_{i=0}^m c_i g^i(x)$.
By the assumption, $g(\lambda)$ commutes with $\lambda$, and so by Corollary \ref{Product}, $g^i(\lambda)=(g(\lambda))^i$ for every $i \in \mathbb{N}$. Hence 
$f^{\circ n}(\lambda)=\sum_{i=0}^m c_i g^i(\lambda)=\sum_{i=0}^m c_i (g(\lambda))^i=f(g(\lambda))=f^{*n}(\lambda)$.
\end{proof}

\begin{thm} \label{thm:periodic}
Let $D$ be a division ring, $f(x) \in D[x]$, $\lambda \in D$ and $r\in\mathbb{N}$ for which $f^{\circ r}(\lambda)=\lambda$ and $\lambda$ commutes with either $\{f^{\circ t}(\lambda) : 1\leq t<r\}$ or $\{f^{*t}(\lambda) : 1\leq t<r\}$.
Then $\lambda$ is an $r$-periodic point of $f(x)$ that commutes with its orbit.
\end{thm}

\begin{proof}
By the previous theorem, $f^{\circ t}(\lambda)=f^{*t}(\lambda)$ for all $t \in \{1,\dots,r\}$. Now, since $f^{* r}(\lambda)=f^{\circ r}(\lambda)=\lambda$, we obtain $f^{* (nr+\ell)}(\lambda)=f^{* \ell}(\lambda)$ for any $n \in \mathbb{N}$ and $\ell \in \{0,\dots,r-1\}$. As a result, $\lambda$ commutes with $\{f^{* t}(\lambda) : t \in \mathbb{N}\}$. By the previous theorem, we conclude that $\lambda$ commutes with its orbit and is $r$-periodic.
\end{proof}

\begin{cor} \label{cor:fixed}
Given a division ring $D$, a polynomial $f(x) \in D[x]$ and $\lambda \in D$ for which $f(\lambda)=\lambda$. Then $\lambda$ is a fixed point of $f(x)$.
\end{cor}

\begin{proof}
Immediate from the previous theorem, given that $r=1$ in this case, the requirement that $\lambda$ commute with $\{f^{\circ t}(\lambda) : 1\leq t<r\}$ is redundant.
\end{proof}

\begin{cor}
Given a division ring $D$ and a polynomial $f(x) \in D[x]$, the fixed points of $f(x)$ are precisely the roots of $g(x)=f(x)-x$.
\end{cor}

\begin{cor}
Given a division ring $D$ and a polynomial $f(x) \in D[x]$ of degree $m \geq 2$, the fixed points of $f(x)$ belong to at most $m$ conjugacy classes in $D$.
\end{cor}

The last corollary is a result of \cite[Theorem 2]{GM} which states that the roots of a polynomial $f(x) \in D[x]$ of degree $n$ group into at most $n$ distinct conjugacy classes. Moreover, \cite[Theorem 4]{GM} states that from each conjugacy class there is either one root or infinitely many, and thus the polynomial $f(x)$ has either infinitely many roots or at most $n$ roots. This means that there is the possibility of having infinitely many distinct fixed points for a standard polynomial over a division ring, something that cannot happen over fields.
For example, the polynomial $f(x)=-x^3 \in \mathbb{H}[x]$ has infinitely many distinct fixed points, including all the elements of the set $\{ai+bj+cij : a^2+b^2+c^2=1\}$. 

Specifically for quaternion algebras $Q$, there is a complete method for solving polynomial equations (see \cite{ChapmanMachen:2017}):
One takes $g(x)=\sum_{i=0}^m c_i x^i$ and computes $C_g(x)=\overline{g(x)} \cdot g(x)$, where $\overline{g(x)}=\sum_{i=0}^m \overline{c_i} x^m$ and $a\mapsto \overline{a}$ stands for the (unique) symplectic involution on $Q$.
The roots of $C_g(x)$ are precisely the conjugacy classes of the roots of $g(x)$.
Therefore, one needs to solve the equation $C_g(x)=0$ over the compositum of the maximal subfields of $Q$.
The roots come in pairs $\lambda,\overline{\lambda}$.
For each such pair, the conjugacy class this pair represents is characterized by $\tr(\lambda)$ and $\norm(\lambda)$.
In order to find the root (or roots) of $g(x)$ from that conjugacy class, one reduces the equation $g(x)=0$ to a linear equation by using the identity $x^2-\tr(\lambda)x+\norm(\lambda)=0$, which holds for all the elements in the conjugacy class of $\lambda$. The linear equation is either trivial (in which case, all the elements from that conjugacy class are roots), or has a unique solution, which is the root of $g(x)$ from that class.

In order to find the fixed points of a polynomial $f(x)$, one needs to carry out this process for $g(x)=f(x)-x$.

\begin{exmpl}
Consider Hamilton's algebra $\mathbb{H}$ of real quaternions, and the polynomial $f(x)=x^2+(i+1)x+1+ij$.
We want to find its fixed points. 
Therefore, we need to find the roots of $g(x)=f(x)-x=x^2+ix+1+ij$.
The polynomial $C_g(x)$ is $x^4+3x^2+2=(x^2+1)(x^2+2)$.
Therefore, the conjugacy classes of its roots are the classes of $i$ and  $\sqrt{2}i$.
The conjugacy class of $i$ is characterized by $x^2+1=0$, and with that, the equation $g(x)=0$ reduces to $-1+ix+1+ij=0$, and so $-ij=ix$, which means $x=-j$ is the unique root of $g(x)$ from that conjugacy class. The second conjugacy class is characterized by $x^2+2=0$, with which the equation $g(x)=0$ reduces to $-2+ix+1+ij=0$, and so $ix=1-ij$, which leaves $x=-i-j$ as the second and last root of $g(x)$.
Thus, the fixed points of $f(x)$ are $-j$ and $-i-j$.
\end{exmpl}

\begin{rem}
If $K$ is a subfield of a division ring $D$, then every periodic point $\lambda \in K$ of $f(x) \in K[x]$ remains a periodic point of $f(x)$ as a polynomial in $D[x]$, because the coefficients of $f(x)$ commute with $\lambda$. For example, $-i$ is a 2-periodic point of $f(x)=x^2+i$ as a polynomial in $\mathbb{H}[x]$.

Note also that even in such a case, there may be fixed points outside of $K$. For example, the polynomial $f(x)=ix^3+(1+i)x$ belongs to $\mathbb{C}[x]$, but as a polynomial in $\mathbb{H}[x]$ it has $j$ as a fixed point.
\end{rem}

\section{Octonion Polynomials}

We conclude the paper with a note on octonion polynomials and fixed points.
Given a field $F$, an octonion algebra $A$ over $F$ is an algebra of the form $A=Q \oplus Q \ell$ where $Q$ is a quaternion $F$-algebra, $\ell^2=\gamma \in F^\times$, and for every $q,r,s,t \in Q$ we have
$$(q+r\ell)(s+t\ell)=qs+\gamma\overline{t}r+(tq+r\overline{s})\ell,$$
where $z\mapsto \overline{z}$ denotes the canonical involution on $Q$.
Taking $F=\mathbb{R}$, $Q=\mathbb{H}$ and $\gamma=-1$, one obtains the classical octonion division algebra over the reals, which is denoted by $\mathbb{O}$.
Octonion algebras satisfy several interesting properties: they are composition algebras (i.e., have a multiplicative quadratic norm form) and alternative (i.e., every two elements generate an associative subalgebra). In fact, they are the largest possible composition algebras by Hurwitz's theorem (\cite[Section 33, Theorem 33.17]{BOI}), and the only possible non-associative alternative division algebras by Kleinfeld's theorem (\cite[Chapter 7, Section 3]{SSSZ}). 
Therefore, it is only natural to ask which properties of associative division rings extend to octonion division algebras as well (see \cite{Chapman:2020}).

The ring of octonion polynomials $A[x]$ is again the scalar extension $A \otimes_F F[x]$, thus $x$ is central in $A[x]$. Moreover, $A[x]$ is an octonion $F[x]$-algebra in the sense of \cite{LoosPeterssonRacine:2008}, and in particular, alternative.
Again, every polynomial can be written as $f(x)=\sum_{i=0}^m c_i x^i$ with the coefficients placed on the left-hand side of the variable, and substitution of $\lambda \in A$ is defined by $f(\lambda)=\sum_{i=0}^m c_i \lambda^i$.
Substituting another polynomial $g(x)$ in $f(x)$ is also given by $f\circ g(x)=f(g(x))=\sum_{i=0}^m c_i(g(x))^i$.
\begin{ques}
Suppose $A$ is an octonion algebra, $f(x) \in A[x]$, $\lambda \in A$, and $f(\lambda)=\lambda$. Does it imply that $f^{\circ n}(\lambda)=\lambda$ for any $n \in \mathbb{N}$?
\end{ques}

The answer is surprisingly negative, as the following example shows:
\begin{exmpl}
Consider $\mathbb{O}$ and the polynomial $f(x)=\ell x^2+(1-i \ell)x+\ell-(i j)\ell$.
Plugging $j$ in $f(x)$ gives $f(j)=j$.
However, $f\circ f(x)=\ell(\ell x^2+(1-i \ell)x+\ell-(i j)\ell)^2+(1-i\ell)(\ell x^2+(1-i \ell)x+\ell-(i j)\ell)+\ell-(i j)\ell=\ell(-x^4+2\ell x^3-2x^2-2i \ell x^2+(2\ell-2(ij)\ell) x-2)+(i+\ell)x^2-2i\ell x+i+j+\ell-(ij)\ell=-\ell x^4-2x^3-(i+\ell)x^2-2(1+ij+i\ell)x+i+j+\ell-(ij)\ell$.
Thus $f\circ f(j)\neq j$.
\end{exmpl}

\section*{Acknowledgements}
The first author acknowledges the receipt of the Chateaubriand Fellowship (969845L) offered by the French Embassy in Israel, and thanks the hospitality of Anne Qu\'{e}guiner-Mathieu, LAGA, and Institut Galil\'{e}e, Universit\'{e} Paris 13, in the fall of 2020, during which a significant part of the work on this paper was done. Both authors thank Patrick Ingram, Boris Kunyavski{\u\i} and the anonymous referee of this article for their helpful comments.
\bibliographystyle{abbrv}

\begin{thebibliography}{1}
	
	\bibitem{BerhuyOggier2013}
	G.~Berhuy and F.~Oggier.
	\newblock {\em An introduction to central simple algebras and their
		applications to wireless communication}, volume 191 of {\em Mathematical
		Surveys and Monographs}.
	\newblock American Mathematical Society, Providence, RI, 2013.
	\newblock With a foreword by B. A. Sethuraman.
	
	\bibitem{Chapman:2020}
	A.~Chapman.
	\newblock Polynomial equations over octonion algebras.
	\newblock {\em J. Algebra Appl.}, 19(6):2050102, 10, 2020.
	
	\bibitem{ChapmanMachen:2017}
	A.~Chapman and C.~Machen.
	\newblock Standard polynomial equations over division algebras.
	\newblock {\em Adv. Appl. Clifford Algebr.}, 27(2):1065--1072, 2017.
	
	\bibitem{GM}
	B.~Gordon and T.~S.~Motzkin.
	\newblock On the zeros of polynomials over division rings.
	\newblock{ \em Trans. Am. Math. Soc.}, 116: 218--226, 1965.
	
	\bibitem{JanovskaOpfer2010}
	D.~Janovsk{\'a} and G.~Opfer.
	\newblock A note on the computation of all zeros of simple quaternionic
	polynomials.
	\newblock {\em SIAM J. Numer. Anal.}, 48(1):244--256, 2010.
	
	\bibitem{BOI}
	M.-A. Knus, A.~Merkurjev, M.~Rost, and J.-P. Tignol.
	\newblock {\em The book of involutions}, volume~44 of {\em American
		Mathematical Society Colloquium Publications}.
	\newblock American Mathematical Society, Providence, RI, 1998.
	\newblock With a preface in French by J. Tits.
	
	\bibitem{LamLeroyOzturk2008}
	T.~Y. Lam, A.~Leroy, and A.~Ozturk.
	\newblock Wedderburn polynomials over division rings. {II}.
	\newblock In {\em Noncommutative rings, group rings, diagram algebras and their
		applications}, volume 456 of {\em Contemp. Math.}, pages 73--98. Amer. Math.
	Soc., Providence, RI, 2008.
	
	\bibitem{LoosPeterssonRacine:2008}
	O.~Loos, H.~P. Petersson, and M.~L. Racine.
	\newblock Inner derivations of alternative algebras over commutative rings.
	\newblock {\em Algebra Number Theory}, 2(8):927--968, 2008.
	
	\bibitem{Silverman:2007}
	J.~H. Silverman.
	\newblock {\em The arithmetic of dynamical systems}, volume 241 of {\em
		Graduate Texts in Mathematics}.
	\newblock Springer, New York, 2007.
	
	\bibitem{SSSZ}
	K.~A. Zhevlakov, A.~M. Slin'ko, I.~P. Shestakov, and A.~I. Shirshov.
	\newblock {\em Rings that are nearly associative}, volume 104 of {\em Pure and
		Applied Mathematics}.
	\newblock Academic Press, Inc. [Harcourt Brace Jovanovich, Publishers], New
	York-London, 1982.
	\newblock Translated from the Russian by Harry F. Smith.
	
\end{thebibliography}

\end{document}